\documentclass[11pt]{amsart}

\usepackage{amsfonts, amstext, amsmath, amsthm, amscd, amssymb} 
\usepackage{mathtools}

\usepackage{float} 
\usepackage{graphics}
\usepackage{caption}
\usepackage{subcaption}
\usepackage{import} 

\usepackage{physics}  
\renewcommand{\vec}[1]{\boldsymbol{#1}} 

\usepackage{microtype}

\usepackage[hidelinks,pagebackref,pdftex]{hyperref}

\usepackage[dvipsnames]{xcolor}
\definecolor{MyCyan}{HTML}{00F9DE}

\usepackage{marginnote}
\long\def\@savemarbox#1#2{\global\setbox#1\vtop{\hsize\marginparwidth 
  \@parboxrestore\tiny\raggedright #2}}
\newcommand{\RR}{\mathbb{R}}  % Reals
\newcommand{\ZZ}{\mathbb{Z}}  % Integers
\newcommand{\TT}{\mathbb{T}}
\newcommand{\calP}{\mathcal{P}}
\newcommand{\calL}{\mathcal{L}}
\renewcommand{\SS}{\mathbb{S}}

\renewcommand{\setminus}{{\smallsetminus}}

\newcommand{\from}{\colon\thinspace} % as in f \from X \to Y

\newtheorem{theorem}{Theorem}[section]
\newtheorem*{theorem*}{Theorem}
\newtheorem{proposition}[theorem]{Proposition}
\newtheorem{lemma}[theorem]{Lemma}

\newtheorem*{namedtheorem}{\theoremname}
\newcommand{\theoremname}{testing}
\newenvironment{named}[1]{\renewcommand{\theoremname}{#1}\begin{namedtheorem}}{\end{namedtheorem}}

\theoremstyle{definition}
\newtheorem{definition}[theorem]{Definition}

\newtheorem{remark}[theorem]{Remark}

\newcommand{\refthm}[1]{Theorem~\ref{Thm:#1}}
\newcommand{\reflem}[1]{Lemma~\ref{Lem:#1}}
\newcommand{\refprop}[1]{Proposition~\ref{Prop:#1}}

\newcommand{\refdef}[1]{Definition~\ref{Def:#1}}

\newcommand{\reffig}[1]{Figure~\ref{Fig:#1}}

\title{A geometric classification of rod complements in the $3$-torus} 

\author{Connie On Yu Hui} 
\address[]{School of Mathematics, Monash University, VIC 3800, Australia } 
\email[]{onyu.hui@monash.edu}

\begin{document} 

\begin{abstract} 
Rod packings are used in crystallography to describe crystal structures with linear or zigzag chains of particles, and each rod packing can be topologically viewed as a collection of disjoint geodesics in the $3$-torus. Hui and Purcell developed a method to study the complements of rods in the $3$-torus with the use of $3$-dimensional geometry and tools from the $3$-sphere, and they partially classified the geometry of some families of rod complements in the $3$-torus.  In this paper, we provide a complete classification of the geometry of all rod complements in the $3$-torus using topological arguments.  
\end{abstract} 

% Put title, author, abstract at the front
\maketitle

\section{Introduction} 
From crystallography, a rod packing is a packing of uniform cylinders, which represent linear or zigzag chains of particles. In 1977, O'Keeffe and Andersson~\cite{OKeeffe-Andersson:RodPCrystalChem} explained that rod packings can be used to succinctly describe some crystal structures, including those that are common and cannot be easily illustrated by other descriptions. Such structures admit $3$-dimensional translational symmetry and therefore can be mathematically described as a link in the $3$-torus. 

In 2001, O'Keeffe \emph{et~al}~\cite{OKeeffeEtAl:CubicRodPackings} described six possible rod packings that can be realized by existing crystalline materials. Rod packings have since appeared in publications from material science \cite{ERH:PeriodicEntanglementII,RosiEtAl:RodPackingsMOF} and biological science~\cite{EH:SwollenCorneocytes}. 

In \cite{HuiPurcell}, Hui and Purcell introduced topology and hyperbolic geometry to the classification of rod packings. They considered certain rod packings as links in the $3$-sphere, and made use of topological, geometric, and computational tools from the study of links in the $3$-sphere to identify hyperbolic and non-hyperbolic families of rod packings. In particular, they showed that five of the six rod packings of O'Keeffe \emph{et~al} in \cite{OKeeffeEtAl:CubicRodPackings} each admits a complete hyperbolic structure. 

In this paper, we expand that work. We provide topological arguments that characterise the geometry of rod complements without using tools from the $3$-sphere. The following is one of the main results. 

\begin{named}{\refthm{MainThreeOrMoreRods}}
Suppose $R_1$, $R_2$, \ldots, $R_n$ are $n$ disjoint rod-shaped circles embedded in $\TT^3$ for some integer $n \geq 3$. 
The $3$-manifold $\TT^3\setminus (R_1\cup R_2 \cup \ldots \cup R_n)$ admits a complete hyperbolic structure if and only if both of the following conditions hold: 
\begin{enumerate}
    \item Three of the rods are linearly independent.
    \item No pair of distinct parallel rods are linearly isotopic in the complement of the other rods in $\TT^3$.  
\end{enumerate} 
\end{named}

We then classify all geodesic links in the $3$-torus into hyperbolic and non-hyperbolic collections, and characterise Seifert fibred rod complements.  

\begin{named}{\refthm{Main2}}
Suppose $R_1, R_2, \ldots, R_n$ are disjoint rod(s) in $\TT^3$ for some positive integer $n$. The $3$-manifold $\TT^3\setminus (R_1\cup R_2\cup \ldots\cup R_n)$ 
\begin{enumerate}
    \item \label{condition:hypIff} admits a complete hyperbolic structure if and only if three of the rods are linearly independent, and no pair of distinct parallel rods are linearly isotopic in the complement of the other rods; and
    \item \label{condition:SFIff} is Seifert fibred if and only if $n=1$ or all rods are parallel; and 
    \item \label{condition:ToroidalIf} is toroidal if the following holds: 
    \begin{enumerate}
        \item all rod(s) span a line or plane; or 
        \item there exist distinct integers $k,l\in\{1,2,\ldots,n\}$ such that $R_k$ and $R_l$ are linearly isotopic in the complement of the other rods. 
    \end{enumerate} 
\end{enumerate} 
\end{named}

\refthm{MainThreeOrMoreRods} and \refthm{Main2} are related to other research works. Recall that rod packings can be viewed as a $3$-periodic link, which is a link with $3$-dimensional translational symmetry. Others have considered 2-periodic links, which are links with $2$-dimensional translational symmetry, or links in the thickened torus. Akimova and Matveev~\cite{Akimova-Matveev:VirtualKnots} tabulate such knots. Champanerkar, Kofman, and Purcell~\cite{CKP:Biperiodic} study the geometry of their complements when the diagrams are alternating. Adams~\emph{et~al}~\cite{AdamsEtAl:LinksInThickenedS} also study the geometry of alternating $2$-periodic links.  

There is also work on the hyperbolic geometry of links in $3$-manifolds other than the $3$-sphere, including the study of virtual links \cite{AdamsEtAl:tgHypVirtualLinks}, and alternating links on other projection surfaces 
\cite{Howie-Purcell:AltLinksOnSurfaces}. These works require the notion of a diagram, which is not required in this paper. 

The $3$-torus is Seifert fibred and any one-rod complement is also Seifert fibred. Cremaschi and Rodr\'{\i}guez-Migueles \cite{Cremaschi-RodriguezMigueles:HypOfLinkCpmInSFSpaces} consider complements of links in Seifert fibred spaces with hyperbolic $2$-orbifold as base space and present sufficient conditions that guarantee their hyperbolicity. See also \cite{Cremaschi-RodriguezMigueles-Yarmola:VolAndFillingMulticurves}. Their work applies to more general Seifert fibred spaces but the sufficient conditions do not cover the cases when there are parallel geodesics. \refthm{MainThreeOrMoreRods} and \refthm{Main2} in this paper do cover such cases. 

% \subsection{Organization}  
% \refsec{Prelim} provides preliminaries. The key results \refthm{MainThreeOrMoreRods} is stated and proved in \refsec{ThreeOrMoreRods}. The geometric classification of all rod complements in the $3$-torus is stated in \refthm{Main2}, the argument of which can be found in \refsec{CompleteClassification}. \refsec{Application} shows some applications.   

\subsection{Acknowledgements}  
The author thanks Jessica Purcell for her encouragement and helpful suggestions on the draft. She thanks Alex He, Ioannis Iakovoglou, and Jos\'{e} Andr\'{e}s Rodr\'{\i}guez-Migueles for helpful conversations during the Dynamics, Foliation, and Geometry III conference held at the MATRIX research institute in Australia; and she thanks Stephan Tillmann for a helpful conversation during AustMS 2023. She also thanks the anonymous reviewer for the helpful suggestions and comments. 

%%%%%%%%%%%%%%%%%%%%%%%%%%%%%%%%%%%%%%%%%%%%%%%%%%%%%%%%%%%%%%%%% NEW Section %%%%%%%%%%%%%%%%%%%%%%%%%%%%%%%%%%%%%%%%%%%%%%%%%%%%%%%%%%%%%%%%
\section{Preliminaries}\label{Sec:Prelim} 

In this paper, the \emph{$3$-torus} $\TT^3$ is viewed as the unit cube $[0,1]^3$ in $\mathbb{R}^3$ with opposite faces glued in the natural way unless otherwise specified. Let $S$ be a one- or two-dimensional submanifold embedded in a $3$-manifold $M$. The symbol $N(S)$ denotes an open tubular neighbourhood of $S$ embedded in $M$. Denote by $\mathcal{P}\from \RR^3  \to \TT^3$
the covering map $\mathcal{P}(x,y,z) \coloneqq ([x-\lfloor x \rfloor], [y - \lfloor y \rfloor], [z - \lfloor z \rfloor])$.  A vector $(x,y,z)$ in $\RR^3$ is called an \emph{integral vector} if and only if $x$, $y$, $z$ are all integers and $\{x,y,z\}\neq \{0\}$.  

\begin{definition} [Rod-shaped circle] \label{Def:Rodsss}  \ 

\begin{enumerate}
\item A \emph{rod} $R$ in the $3$-torus is the image set $\mathcal{P}(L)$ of a straight line $L$ in $\RR^3$. We call $R$ a \emph{rod-shaped circle} if the set $R=\mathcal{P}(L)$ is a circle embedded in $\TT^3$. By abuse of terminology, we sometimes call a rod-shaped circle simply by a \emph{rod} when the context is clear. 

\item A \emph{lift of a rod $R$} is a connected component of the pre-image set $\calP^{-1}(R)$ under the projection map $\calP\from \RR^3 \to \TT^3$.  

\item \label{Def:uvwRod}
 Let $u$, $v$, $w$ be real numbers, not all zero. Denote by $L$ a straight line in $\RR^3$ that has the same direction as the vector $(u,v,w)$. We call the rod $R=\calP(L)$ a \emph{$(u,v,w)$-rod}. Unless otherwise specified, if $(u,v,w)$ is an integral vector, we assume each $(u,v,w)$-rod is associated with the parametrization $R_e\from [0,1]\to \TT^3$ defined as 
 \[R_e\coloneqq \calP\circ R_\ell,\] 
 where $R_\ell\from [0,1]\to \RR^3$ is the linear map 
 \[R_\ell(t) \coloneqq (u_0, v_0, w_0) + t(u,v,w)\] 
 for some $(u_0, v_0, w_0) \in L$.
\end{enumerate}
\end{definition} 

%%%%%%%%%%%%%%%%%%%%%%%%%%%%%%%
\begin{definition} [Plane-shaped torus] \label{Def:PlaneTori}  \ 

\begin{enumerate}
\item A \emph{plane} $T$ in the $3$-torus is the image set $\mathcal{P}(\Pi)$ of a plane $\Pi$ in $\RR^3$. If the set $T=\mathcal{P}(\Pi)$ is a torus embedded in $\TT^3$, we call $T$ a \emph{plane-shaped torus}, or simply a \emph{plane torus}. 

\item A \emph{lift of a plane torus $T$} is a connected component of the pre-image set $\calP^{-1}(T)$ under the projection map $\calP\from \RR^3 \to \TT^3$.  

\item Let $(a,b,c)$ and $(u,v,w)$ be nonzero, nonparallel vectors in $\RR^3$. A \emph{plane  $T$ spanned by $(a,b,c)$ and $(u,v,w)$ in the $3$-torus} is a plane in $\TT^3$ such that the associated Euclidean plane $\Pi$ is spanned by the vectors $(a,b,c)$ and $(u,v,w)$. In case $(a,b,c)$ and $(u,v,w)$ are integral vectors, such a torus $T$ is associated with the parametrization 
$T_e\from [0,1]\times[0,1]\to \TT^3$ defined as 
\[T_e\coloneqq \calP\circ T_\ell,\] 
where $T_\ell\from [0,1]\times[0,1]\to \RR^3$ is the linear map 
\[T_\ell(s,t) \coloneqq (a_0, b_0, c_0) + s(a,b,c) + t(u,v,w)\] 
for some $(a_0, b_0, c_0) \in \Pi$. 

\end{enumerate}
\end{definition}

\begin{definition}
Let $R_1$, $R_2$, \ldots, $R_m$ be $m$ rod-shaped circles embedded in the $3$-torus.  For each $i\in\{1,2,\ldots,m\}$, denote by $(a_i,b_i,c_i)$ a vector in $\RR^3$ such that $R_i$ is an $(a_i,b_i,c_i)$-rod. We say the rods $R_1$, $R_2$,\ldots, $R_m$ are \emph{linearly independent} if and only if the set of vectors $\{(a_i,b_i,c_i): i \in [1,m]\cap\ZZ\}$ are linearly independent. Two rods $R_1$ and $R_2$ are said to be \emph{parallel} if and only if $(a_1, b_1, c_1)$ is a scalar multiple of $(a_2, b_2, c_2)$. 
\end{definition} 

\begin{definition}
Given a torus $\TT^2$ and a framing $f\from\TT^2 \to \mathbb{S}^1 \times \mathbb{S}^1 \cong ([0,1] /\sim) \times ([0,1] / \sim)$, we can define a covering map $\mathcal{P}:\mathbb{R}^2\xrightarrow{} \TT^2$ by $\mathcal{P}(x,y)= f^{-1}([x-\lfloor x \rfloor], [y - \lfloor y \rfloor])$. Let $p, q \in \mathbb{Z}$, not both zero. A \emph{$(p,q)$-curve in $\TT^2$} is the projection $\mathcal{P}(L)$, where $L$ is a line segment between $(0,0)$ and $(p,q)$, or any other line segment with the same length and slope.  A $(1,0)$-curve is called a \emph{meridian} of the torus $\TT^2$. A $(0,1)$-curve is called a \emph{longitude} of the torus $\TT^2$.  

Let $p$, $q$ be integers, not both zero. 
A \emph{$(p,q)$-curve of a rod-shaped circle $R$} in $\TT^3$ is a $(p,q)$-curve in $\partial \overline{N(R)}$ with the natural framing. 
\end{definition}  

Recall the lemma from \cite{HuiPurcell} that characterises a rod-shaped circle in $\TT^3$:  

\begin{lemma} \label{Lem:characteriseSSC}  %subclaim (S2)
Let $(a,b,c)\in \ZZ^3\setminus\{(0,0,0)\}$. The parametrization $R_e\from [0,1]\to\TT^3$ corresponding to the $(a,b,c)$-rod $R$ represents a simple closed curve in $\TT^3$ if and only if $\gcd(a,b,c)=1$. 
\end{lemma}  

Next, we present the following lemma that characterises a plane-shaped torus in $\TT^3$. 

\begin{lemma} \label{Lem:characterisePlaneTorusAsSet}
A plane $P$ in $\TT^3$ is a plane torus if and only if $P$ is spanned by two nonparallel integral vectors in $\TT^3$. 
\end{lemma}

\begin{proof} 
Suppose the plane $P$ is spanned by two nonparallel integral vectors $(a,b,c)$ and $(u,v,w)$ in $\TT^3$. Without loss of generality, we assume $\gcd(a,b,c)=1$ and $\gcd(u,v,w)=1$. 

If we view the plane $P$ as a set, we have $P$ contains an $(a,b,c)$-rod and a $(u,v,w)$-rod. By \reflem{characteriseSSC}, these rods are simple closed curves in $\TT^3$. 

Next, we consider lattices of vectors in $\RR^3$. Denote by $(a,b,c)_{\vec{p}}$ a copy of the vector $(a,b,c)$  with vector tail sitting at some point $\vec{p}\in \ZZ^3 \subseteq \RR^3$.  Denote the lattice of each spanning vector by $\calL_1\coloneqq \{(a,b,c)_{\vec{p}}: \vec{p}\in \ZZ^3 \}$ and $\calL_2\coloneqq\{(u,v,w)_{\vec{q}}: \vec{q}\in \ZZ^3 \}$ respectively. (Since the intersection pattern is the same in each unit cube with integral vertices, it suffices to consider $\vec{p}=(0,0,0)$.) 

Case 1: Suppose the vector $(a,b,c)_{\vec{p}}$ does not intersect any $(u,v,w)_{\vec{q}}\in \calL_2$ except at tip or tail. Then $P$ is equal to the product between an $(a,b,c)$-rod and a $(u,v,w)$-rod. In other words, $P$ is homeomorphic to $\SS^1\times \SS^1$. Thus, $P$ is a plane torus.  

Case 2: Suppose an interior point $\vec{r}$ of the vector $(a,b,c)_{\vec{p}}$ intersects some $(u,v,w)_{\vec{q}}\in \calL_2$ (see \reffig{GridAndArrows} for example). Observe that $\vec{p}\neq \vec{q}$ and the two positioned vectors $(a,b,c)_{\vec{p}}$ and $(a,b,c)_{\vec{q}}$ lie in the same plane $\Pi$ parallel to any lift of $P$. The plane $\Pi$ in $\RR^3$ contains more than one positioned vector in $\calL_1$; we can thus find an integral vector $(x,y,z)$ with tail positioned at $\vec{p}$ and tip pointing at the tail of a neighbouring positioned vector $(a,b,c)_{\vec{\hat{p}}}$ in $\Pi$. Note that the integral vector $(x,y,z)$ does not intersect any positioned vector in $\calL_1$ except tip and tail. As a result, the plane $P$, which has lift parallel to $\Pi$, contains an $(a,b,c)$-rod and an $(x,y,z)$-rod which are linearly independent and intersect only once.  Hence $P$ is homeomorphic to $\SS^1\times\SS^1$. 

In any case, $P$ is a plane torus. 

Now suppose the plane $P$ in $\TT^3$ is a plane torus. 
Note that $P$ is homeomorphic to a torus $\SS^1\times\SS^1$, it is the product of two rod-shaped circles $R_1$ and $R_2$ which represent two distinct generators in $\pi_1(P)\cong\ZZ\times\ZZ$. Observe that any rod-shaped circle is parallel to some integral vector. Denote by $\vec{v_1}$ and $\vec{v_2}$ the integral vectors that $R_1$ and $R_2$ are respectively parallel to. Since $R_1$ and $R_2$ are nonparallel, $\vec{v_1}$ and $\vec{v_2}$ are nonparallel. Therefore, $P$ is spanned by two nonparallel integral vectors. 
\end{proof} 

\begin{figure} 
%% Creator: Inkscape 1.1.2 (b8e25be833, 2022-02-05), www.inkscape.org
%% Accompanies image file 'GridAndArrows.pdf' (pdf, eps, ps)
\begingroup%
  \makeatletter%
  \providecommand\color[2][]{%
    \errmessage{(Inkscape) Color is used for the text in Inkscape, but the package 'color.sty' is not loaded}%
    \renewcommand\color[2][]{}%
  }%
  \providecommand\transparent[1]{%
    \errmessage{(Inkscape) Transparency is used (non-zero) for the text in Inkscape, but the package 'transparent.sty' is not loaded}%
    \renewcommand\transparent[1]{}%
  }%
  \providecommand\rotatebox[2]{#2}%
  \newcommand*\fsize{\dimexpr\f@size pt\relax}%
  \newcommand*\lineheight[1]{\fontsize{\fsize}{#1\fsize}\selectfont}%
  \ifx\svgwidth\undefined%
    \setlength{\unitlength}{219.99999685bp}%
    \ifx\svgscale\undefined%
      \relax%
    \else%
      \setlength{\unitlength}{\unitlength * \real{\svgscale}}%
    \fi%
  \else%
    \setlength{\unitlength}{\svgwidth}%
  \fi%
  \global\let\svgwidth\undefined%
  \global\let\svgscale\undefined%
  \makeatother%
  \begin{picture}(1,0.27805076)%
    \lineheight{1}%
    \setlength\tabcolsep{0pt}%
    \put(0,0){\includegraphics[width=\unitlength,page=1]{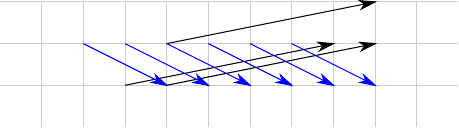}}%
    \put(0.25507759,0.04542828){\makebox(0,0)[lt]{\lineheight{1.25}\smash{\begin{tabular}[t]{l}$p$\end{tabular}}}}%
    \put(0.35053086,0.03861009){\makebox(0,0)[lt]{\lineheight{1.25}\smash{\begin{tabular}[t]{l}$\hat{p}$\end{tabular}}}}%
    \put(0.34258228,0.20150037){\color[rgb]{0,0,1}\makebox(0,0)[lt]{\lineheight{1.25}\smash{\begin{tabular}[t]{l}$q$\end{tabular}}}}%
    \put(0,0){\includegraphics[width=\unitlength,page=2]{GridAndArrows.pdf}}%
  \end{picture}%
\endgroup%
  
\caption{A particular example illustrating Case 2 in the proof of \reflem{characterisePlaneTorusAsSet}. }
\label{Fig:GridAndArrows} 
\end{figure}  

\reflem{characterisePlaneTorusAsSet} tells us exactly when a plane in the $3$-torus is a plane torus. Observe that every plane torus, as a set, is a torus embedded in the $3$-torus.

If we consider more specifically the associated parametrizations of the torus, the following lemma states the condition that characterises a parametrization representing a torus embedded in the $3$-torus.  

\begin{lemma} \label{Lem:characterisePlaneTorus}
Let $v_1 = (a,b,c)$, $v_2 = (m,n,r)\in \ZZ^3\setminus\{(0,0,0)\}$. Suppose $v_1$ and $v_2$ are linearly independent. The parametrization $T_e\from [0,1]\times[0,1]\to \TT^3$ associated to the plane $T$ spanned by $v_1$ and $v_2$ represents an embedded plane torus in $\TT^3$ if and only if \[\gcd\bigg(\mdet{a & m \\ b & n}, \mdet{b & n \\ c & r}, \mdet{a & m \\ c & r}\bigg)=1.\] 
\end{lemma}

\begin{proof}
Suppose $T_e$ represents an embedded plane torus in $\TT^3$.  Then $T_e|_{[0,1]\times\{0\}}$ and $T_e|_{\{0\}\times[0,1]}$ represent two simple closed curves. As the distance between two neighbouring lifts of $T_e([0,1]\times[0,1])$ is nonzero, there exists a nonzero vector $v_3 = (\alpha, \beta, \gamma) \in\ZZ^3$ such that $v_1$, $v_2$, $v_3$ bound a fundamental region of the $3$-torus. Note that $\operatorname{MCG}(\TT^3) \cong \operatorname{GL}_3(\ZZ)$. We thus have 
\begin{align*}
\det(v_1^\top, v_2^\top, v_3^\top) = \pm 1, \textrm{ or } \alpha \mdet{b & n \\ c & r} - \beta \mdet{a & m \\ c & r} + \gamma \mdet{a & m \\ b & n} = \pm 1. 
\end{align*} 

The last equality implies that 
\[\gcd\bigg(\mdet{a & m \\ b & n}, \mdet{b & n \\ c & r}, \mdet{a & m \\ c & r}\bigg)=1.\]

Next, we show the converse. Suppose 
$\gcd\bigg(\mdet{a & m \\ b & n}, \mdet{b & n \\ c & r}, \mdet{a & m \\ c & r}\bigg)=1.$ 
By the generalised B\'{e}zout's lemma, there exist integers $\alpha$, $\beta$, $\gamma$ such that
\begin{align*}
\alpha \mdet{b & n \\ c & r} - \beta \mdet{a & m \\ c & r} + \gamma \mdet{a & m \\ b & n} = 1, \textrm{ or }
\mdet{
a & m & \alpha \\
b & n & \beta \\
c & r & \gamma} = 1
\end{align*} 

Define $v_3\coloneqq (\alpha, \beta, \gamma)$. The matrix $(v_1^\top, v_2^\top, v_3^\top)$ belongs to $\operatorname{GL}_3(\ZZ)$. Hence there is a homeomorphism $h$ mapping $(1,0,0)$, $(0,1,0)$, $(0,0,1)$ to $v_1$, $v_2$, $v_3$ respectively. Since the parametrization of the torus spanned by $(1,0,0)$ and $(0,1,0)$ is an embedding in $\TT^3$, its image set under $h$ is also a torus embedded in $\TT^3$. Hence, $T_e$ represents an embedded plane torus in $\TT^3$.   
\end{proof}

\begin{lemma} \label{Lem:PlaneTorusIsEssential}
Let $L$ be a link embedded in the $3$-torus $\TT^3$. 
If $T$ is a plane torus embedded in $\TT^3 \setminus L$, then $T$ is an essential torus in $\TT^3 \setminus L$. 
\end{lemma} 

\begin{proof}
Since $T$ is a plane torus, any two rods that represent the two generators of $\pi_1(T) \cong \ZZ\times \ZZ$ also represent distinct generators of a subgroup of $\pi_1(\TT^3) \cong \ZZ\times\ZZ\times\ZZ$. The two rods thus represent distinct nontrivial elements in $\pi_1(\TT^3\setminus L)$ with infinite order, otherwise, the rods can be ambient isotoped to one another in $\TT^3$ or one of them can be homotoped to a point in $\TT^3$. Hence, $T$ is $\pi_1$-injective, and thus incompressible in $\TT^3$. 

As $T$ is non-separating in $\TT^3\setminus L$, the surface $T$ cannot be boundary-parallel in $\TT^3\setminus L$. Since $T$ is incompressible and not boundary-parallel, $T$ is essential in $\TT^3\setminus L$.  
\end{proof} 

\begin{remark}
An argument similar to the proof of \reflem{PlaneTorusIsEssential} can be used to show that any plane torus in $\TT^3$ is essential. 
\end{remark}

The following is Theorem 4.1 from \cite{HuiPurcell}. 

\begin{theorem} \label{Thm:MainSingleRod_TwoRods}
Let $R_1$ and $R_2$ be rod-shaped circles embedded in $\TT^3$.
\begin{itemize}
\item The $3$-manifold $\TT^3\setminus R_1$ is Seifert fibred. 

\item If $R_1$ and $R_2$ are parallel rods, then $\TT^3\setminus(R_1\cup R_2)$ is Seifert fibred.
\item If $R_1$ and $R_2$ are linearly independent rods, then $\TT^3\setminus(R_1\cup R_2)$ is toroidal.
\end{itemize}
Any complement of one or two rods in $\TT^3$ cannot admit a complete hyperbolic structure. 
\end{theorem} 

%%%%%%%%%%%%%%%%%%%%%%%%%%%%%%%%%%%%%%%%%%%%%%%%%%%%%%%%%%%%%%%%% NEW Section %%%%%%%%%%%%%%%%%%%%%%%%%%%%%%%%%%%%%%%%%%%%%%%%%%%%%%%%%%%%%%%%
\section{Three or more arbitrary rods in the $3$-torus}\label{Sec:ThreeOrMoreRods} 

A result in \cite{HuiPurcell} states that each member of a $4$-rod family of links in the $3$-torus admits a complete hyperbolic structure. That paper used tools from links in the $3$-sphere to prove the result. In the following, we use a topological argument to show a more general result.  

\begin{theorem} \label{Thm:MainThreeOrMoreRods} %%% = Claim (S7) 
Let $n$ be an integer greater than or equal to three. 
Suppose $R_1$, $R_2$, \ldots, $R_n$ are $n$ disjoint  rod-shaped circles embedded in $\TT^3$. 
The $3$-manifold $\TT^3\setminus (R_1\cup R_2 \cup \ldots \cup R_n)$ admits a complete hyperbolic structure if and only if both of the following conditions hold: 
\begin{enumerate}
    \item There exist three linearly independent rods in $\{R_1, R_2, \ldots, R_n\}$. 
    \item For each pair of distinct parallel rods $R_i$ and $R_j$, they are not linearly isotopic in $\TT^3\setminus ((R_1\cup R_2 \cup \ldots \cup R_n)\setminus (R_i \cup R_j))$.       
\end{enumerate} 
\end{theorem} 

The remainder of this section will focus on the proof of \refthm{MainThreeOrMoreRods}. 

\subsection{Necessary conditions for hyperbolic rod complements} 
\begin{definition}
Let $R_i$ and $R_j$ be two disjoint parallel rods embedded in a $3$-manifold $M$ which is  $\TT^3$ or a link complement in $\TT^3$. 
The rods $R_i$ and $R_j$ are said to be \emph{linearly isotopic} in $M$ if and only if there exist parametrizations $\gamma_i, \gamma_j\from\SS^1\to \TT^3$ for $R_i, R_j$ and an isotopy $H\from \SS^1 \times [0,1] \to M$ such that $H_0 = \gamma_i$, $H_1 = \gamma_j$, and there exists a vector $\vec{v}\in\RR^3$ such that for each $s\in\SS^1$, for each $t\in [0,1]$, $H(s,t) = \gamma_i(s) + t\vec{v}$  (modulo $(1,1,1)$). We call $H$ a \emph{linear isotopy} between $R_i$ and $R_j$. 
\end{definition}

Let $n$ be an integer greater than or equal to three. 
Suppose $R_1$, $R_2$, $\ldots$, $R_n$ are $n$ disjoint  rod-shaped circles embedded in the $3$-torus $\TT^3$. Denote by $M$ the compact $3$-manifold $\TT^3\setminus N(R_1\cup R_2 \cup \ldots \cup R_n)$ obtained by removing disjoint tubular neighbourhoods of the rods from $\TT^3$. 

Note that if $\TT^3\setminus (R_1\cup R_2 \cup \ldots \cup R_n)$ admits a complete hyperbolic structure, then $M$ is atoroidal and anannular by Thurston's hyperbolization theorem. The following two propositions state the necessary conditions. 

\begin{proposition} \label{Prop:HyperbolicImpliesLI}
If the $3$-manifold $\TT^3\setminus (R_1\cup R_2 \cup \ldots \cup R_n)$ admits a complete hyperbolic structure, then there exist three linearly independent rods in $\{R_1, R_2,\ldots, R_n\}$. 
\end{proposition} 

\begin{proof}
Suppose not. Then there exists a $2$-dimensional linear subspace $P$ of $\RR^3$ such that $P$ is spanned by two nonparallel integral vectors and each rod $R_i$ is parallel to some vector $\vec{v_i} \in P$. 

Let $k\in\{1,2,\ldots,n\}$. Each connected component of $\mathcal{P}^{-1}(R_k)$ is a straight line  $L_k^{(j)}$ lying in a plane $\Pi_k^{(j)}$ parallel to $P$ in $\RR^3$.  Denote by $\{L_k^{(j)}: j\in J\}$ the set of all connected components of $\mathcal{P}^{-1}(R_k)$.  
By \reflem{characterisePlaneTorusAsSet}, $\mathcal{P}(\Pi_k^{(j)})$ is a torus $T_k$ embedded in $\TT^3$ for each $j\in J$.  Since each rod is an embedded circle that intersects each of the six faces of the unit cube representing $\TT^3$ finitely many times, the union of rods $R_1$, $R_2$, \ldots, $R_n$ intersect each of the faces of the unit cube finitely many times as well. Therefore, the pre-image set $\bigcup_{k=1}^n \mathcal{P}^{-1}(T_k)$ has finitely many planes that intersect each unit cube in $\RR^3$. Hence,
\[ \min\{\hspace{0.5mm} d(\Pi_k^{(j)}, \Pi_l^{(i)}): k,l\in\{1, 2,\ldots, n\} \textrm{ and } j, i\in J \textrm{ and } \Pi_k^{(j)}\neq \Pi_l^{(i)} \} > 0. \]

As the minimal distance between planes is nonzero, we can always find another plane $\tilde{\Pi}$ parallel to $P$ that is disjoint from all other planes $\Pi_k^{(j)}$ in $\RR^3$. By \reflem{PlaneTorusIsEssential}, the projection $\mathcal{P}(\tilde{\Pi})$ would be  
an essential torus in 
$\TT^3\setminus (R_1\cup R_2 \cup \ldots \cup R_n)$, contradicting the fact that $M$ is atoroidal. Therefore, there exist three linearly independent rods in $\{R_1, R_2, \ldots, R_n\}$. 
\end{proof}

\begin{remark}
Note that it is possible that $\Pi_k^{(j)} = \Pi_k^{(i)}$ for some distinct $j,i \in J$, that is, the plane $\Pi_k^{(j)}$ may contain more than one connected component  of $\mathcal{P}^{-1}(R_k)$. It is also possible that  $\Pi_k^{(j)} = \Pi_l^{(i)}$ for some distinct $k,l \in \{1,\ldots,n\}$ when $R_k$ and $R_l$ are parallel. 
\end{remark}

% Before stating another necessary condition, let us quote a known lemma below. 

% \begin{lemma} \label{Lem:AnnulusBIiffNBP} %subclaim (S1) 
% Let $M$ be a compact $3$-manifold with nonempty boundary consisting only of tori.
% Suppose $M$ is neither a solid torus nor a thickened torus. 
% If $M$ is irreducible and boundary-irreducible, 
% then any properly embedded annulus $A$ in $M$ is 
% $\partial$-incompressible if and only if $A$ is not boundary-parallel. 
% \end{lemma} 

\begin{proposition} \label{Prop:HyperbolicImpliesNotLinIsotopic}
If the $3$-manifold $\TT^3\setminus (R_1\cup R_2 \cup \ldots \cup R_n)$ admits a complete hyperbolic structure, then for each pair of distinct parallel rods $R_i$ and $R_j$, there is no linear isotopy between $R_i$ and $R_j$ in $\TT^3\setminus ((R_1\cup R_2 \cup \ldots \cup R_n)\setminus (R_i \cup R_j))$. 
\end{proposition} 

\begin{proof} 
Suppose not. The linear isotopy gives an annulus $A$ properly embedded in $M = \TT^3\setminus N(R_1\cup R_2 \cup \ldots \cup R_n)$ with $\partial A$ equal to the union of two disjoint $(0,1)$-curves in $\partial N(R_i)$ and $\partial N(R_j)$ respectively. Since the two components of $\partial A$ lie in distinct boundary components of $M$, the annulus $A$ cannot be boundary-parallel in $M$. 

Because $R_i$ and $R_j$ are parallel, they both represent the same element $(p,q,r) \in \ZZ\times\ZZ\times\ZZ \cong \pi_1(\TT^3)$ with infinite order. The generator of $\pi_1(A)$ also corresponds to this nontrivial element $(p,q,r)$. Then the annulus $A$ is $\pi_1$-injective, and thus incompressible in $M$. 

Therefore, $A$ would be an essential annulus properly embedded in $M$, contradicting the fact that $M$ is anannular.  % Hence, each pair of distinct parallel rods $R_i$ and $R_j$ are not linearly isotopic in $\TT^3\setminus ((R_1\cup R_2 \cup \ldots \cup R_n)\setminus (R_i \cup R_j))$. 
\end{proof} 

\subsection{Sufficient conditions for hyperbolic rod complements} \ 

In this subsection, we show that the necessary conditions stated in \refprop{HyperbolicImpliesLI} and \refprop{HyperbolicImpliesNotLinIsotopic} are also sufficient.

\begin{proposition} \label{Prop:NRodsIrredAndBIrred}
The complement of any (possibly empty) union of finitely many rods in the $3$-torus is irreducible and $\partial$-irreducible. 
\end{proposition} 

\begin{proof}
Let $R_1$, $R_2$, \ldots, $R_n$ be rod-shaped circles in $\TT^3$ for some positive integer $n$. Let $\hat{R_i}=\mathcal{P}^{-1}(R_i)$ for each $i\in [1,n]\cap\ZZ$. Denote by 
\[\mathcal{P}\from \RR^3 \setminus (\hat{R_1} \cup \hat{R_2} \cup \cdots \cup \hat{R_n}) \to \TT^3 \setminus(R_1\cup R_2 \cup \cdots \cup R_n)\]  
the covering map $\mathcal{P}(x,y,z) \coloneqq ([x-\lfloor x \rfloor], [y - \lfloor y \rfloor], [z - \lfloor z \rfloor])$. \vspace{1mm}

Suppose there were an essential sphere $S$ in $\mathbb{R}^3 \setminus (\hat{R_1} \cup \hat{R_2} \cup \cdots \cup \hat{R_n})$.  By Alexander's theorem \cite[Theorem 1.1]{Hatcher:NotesBasic3MfldTop}, $S$ bounds an embedded $3$-ball $B$ in $\mathbb{R}^3$. Since $S$ is essential, the interior of the $3$-ball $B$ will contain a straight line $L \subseteq (\hat{R_1} \cup \hat{R_2} \cup \cdots \cup \hat{R_n})$, contradicting the fact that $B$ is bounded. Therefore, $\mathbb{R}^3 \setminus (\hat{R_1} \cup \hat{R_2} \cup \cdots \cup \hat{R_n})$ is irreducible. Since spaces covered by irreducible spaces are irreducible  \cite[Proposition 1.6]{Hatcher:NotesBasic3MfldTop}, $\TT^3 \setminus(R_1\cup R_2 \cup \cdots \cup R_n)$ is irreducible. % Hence $M \coloneqq \TT^3\setminus N(R_1\cup R_2 \cup \cdots \cup R_n)$ is also irreducible. 

Suppose there were a disc $D$ properly embedded in $M\coloneqq \TT^3\setminus N(R_1\cup R_2 \cup \cdots \cup R_n)$ such that $\partial D$ does not bound a disc in $\partial M$. As $\partial D$ is a connected set and does not bound a disc in $\partial M$, $\partial D$ would be a non-trivial $(p,q)$-curve in $\partial N(R_i)$ for some $i\in\{1,2,\ldots,n\}$. Note also that $R_i$ is a rod-shaped circle in $\TT^3$; it represents a non-trivial loop class in $\pi_1(\TT^3) \cong \mathbb{Z}\times\mathbb{Z}\times\mathbb{Z}$. Thus any $(p,q)$-curve of $\partial N(R_i)$ with $q\neq 0$ cannot bound a disc in $\TT^3\setminus N(R_1\cup R_2\cup \cdots \cup R_n)$. 

Therefore, $\partial D$ would be a $(p,0)$-curve of $\partial N(R_i)$.  Since $\partial D$ is embedded in $\partial N(R_i)$, it has to be a $(1,0)$-curve (i.e., meridian) of $\partial N(R_i)$.  Note that each connected component of $\mathcal{P}^{-1}(D)$ is a disc $\tilde{D}$ with boundary equal to a meridional circle of an infinite straight line $\tilde{R}$ in $\RR^3$. The union of $\tilde{D}$ and a meridional disc of $N(\tilde{R})$ bounded by $\partial \tilde{D}$ is an embedded $2$-sphere in $\RR^3$ that intersects $\tilde{R}$ at exactly one point transversely.  This contradicts the fact that an embedded $2$-sphere in $\RR^3$ bounds a $3$-ball which cannot contain an infinite half line. 

Hence,  $M\coloneqq \TT^3\setminus N(R_1\cup R_2\cup \cdots \cup R_n)$ is $\partial$-irreducible.  
\end{proof} 

Next, we show that the two necessary conditions in \refprop{HyperbolicImpliesLI} and \refprop{HyperbolicImpliesNotLinIsotopic} suffice to guarantee that there are no essential tori nor essential annuli in the rod complement. 

\begin{lemma}
\label{Lem:ExistsCompressionDisc}
    If $T$ is an embedded torus in $\TT^3$ with a generator of $\pi_1(T) \cong \ZZ\times\ZZ$ represented by a loop that is homotopically trivial in $\TT^3$, then there exists a compression disc for $T$ in $\TT^3$. 
\end{lemma} 

\begin{proof}
Suppose $\gamma_h\subset T$ is a homotopically trivial loop in $\TT^3$ such that $[\gamma_h]$ is a generator of $\pi_1(T)$. Note that any lift $\tilde{T}$ of $T$ is a covering space of $T$ corresponding to the kernel of the induced map $\pi_1(T)\to\pi_1(\TT^3)$ (see \cite[Item~15 on p.80]{Hatcher:AlgTop}), which is a nontrivial subgroup of $\pi_1(T)$ because of the assumption on $\gamma_h$. It follows from the Galois correspondence \cite[Theorem 1.38]{Hatcher:AlgTop} and the classification of possibly punctured surfaces that any lift $\tilde{T}$ of $T$ is a torus or an infinite cylinder embedded in $\RR^3$. 

Case 1: Suppose $\tilde{T}$ is a torus. Using an argument similar to the proof for Alexander's theorem \cite[Theorem 1.1]{Hatcher:NotesBasic3MfldTop}, we can show that the torus $\tilde{T}$ bounds a (possibly once-punctured) solid torus $V$ in $\RR^3$ (see \cite[Item 3 on p.9]{Hatcher:NotesBasic3MfldTop}). By the definition of the covering map $\calP \from \RR^3\to \TT^3$, we can pick a meridional disc $\tilde{D}$ in $V\setminus \calP^{-1}(N(T))$ such that $\calP|_{\tilde{D}}\from \tilde{D}\to\TT^3\setminus N(T)$ is a map with no singularities on $\partial\tilde{D}$. By Dehn's lemma \cite{Rolfsen:KnotsAndLinks}, there exists an embedding $g\from \tilde{D}\to \TT^3\setminus N(T)$ with $g(\partial\tilde{D})=\calP|_{\tilde{D}}(\partial\tilde{D})$. Hence there exists a compression disc for $T$ in $\TT^3$. 

Case 2: Suppose $\tilde{T}$ is an infinite cylinder.  We can find a closed $3$-ball $B_c$ embedded in $\TT^3$ such that $\gamma_h$ is a subset of the interior of $B_c$. Observe that we can isotope the $3$-ball $B_c$ such that $\partial B_c$ intersects $T$ transversely and $\partial B_c\cap T$ is a collection of disjoint circles $\{C_i\}$. By a consequence of Schoenflies theorem \cite{Rolfsen:KnotsAndLinks}, each $C_i$ bounds a disc $D(C_i,\partial B_c)$ in $\partial B_c$. 

Suppose it were true that each circle $C_i$ bounds a disc $D(C_i,T)$ in $T$. We first consider the circles $C_i$'s with $D(C_i,T)$'s not containing any other $C_j$'s (i.e. consider the innermost circles). By irreducibility of $\TT^3$, the embedded spheres $D(C_i,\partial B_c) \cup D(C_i,T)$'s each bounds a $3$-ball in $\TT^3$. Thus, we could ambient-isotope $T$ (by deforming $D(C_i,T)$ to $D(C_i,\partial B_c)$ in the $3$-ball and then pushing $D(C_i,T)$ further away from $D(C_i,\partial B_c)$) such that all these innermost circles $C_i$'s disappear in $\partial B_c$. We repeat the process for each set of innermost circles consecutively until all $C_i$'s disappear. 

As $\gamma_h$ is homotopically nontrivial in $T$, it cannot lie in any disc $D(C_i,T)$, thus it cannot be pushed away from $B_c$. However, the deformed $T$ would be disjoint from $\partial B_c$, implying that $\tilde{T}$ would be either disconnected or bounded within the lift $\tilde{B_c}$ of $B_c$, a contradiction. Hence, there exists a loop $C_k$ embedded in $T$ that bounds a compression disc $D(C_k,\partial B_c)$ for $T$ in $\TT^3$.

In either case, we obtain a compression disc for $T$ in $\TT^3$.    
\end{proof}

\begin{proposition} \label{Prop:LINIImpliesAtoroidal}
Suppose there exist three linearly independent rods in $\{R_1$, $R_2$, \ldots , $R_n\}$ and for each pair of distinct parallel rods $R_i$ and $R_j$, they are not linearly isotopic to each other in $\TT^3\setminus ((R_1\cup R_2 \cup \ldots \cup R_n)\setminus (R_i \cup R_j))$.  Then the $3$-manifold $\TT^3\setminus (R_1\cup R_2 \cup \ldots \cup R_n)$ is atoroidal. 
\end{proposition} 

\begin{proof} 
Let $R_{\alpha}$, $R_{\beta}$, and $R_{\gamma}$ be three linearly independent rods in $\{R_1$, $R_2$, \ldots , $R_n\}$.  Suppose it were true that there exists an essential torus $T_e$ in $\TT^3\setminus (R_1\cup R_2 \cup \ldots \cup R_n)$. 

If $T_e$ were isotopic to a plane torus in $\TT^3$, $T_e$ would intersect at least one of the three rods $R_{\alpha}$, $R_{\beta}$, and $R_{\gamma}$, a contradiction. Thus, $T_e$ is not isotopic to any plane torus. Hence, at least one generator in $\pi_1(T_e)\cong\ZZ\times\ZZ$ is represented by an embedded loop $\gamma_h$ that is homotopically trivial in $\TT^3$. By \reflem{ExistsCompressionDisc}, there exists a compression disc $D_m$ for $T_e$ in $\TT^3$. 

We can then surger $T_e$ along $D_m$ to obtain a $2$-sphere $S$ embedded in $\TT^3$. By \refprop{NRodsIrredAndBIrred}, the $2$-sphere $S$ bounds a $3$-ball $B$ in $\TT^3$. % Note that each of $(D_m)_1$ and $(D_m)_2$ has two sides, let us call the side that ``touches the knife'' the cut-side, and call the other side the non-cut-side. We have either both cut-sides facing $B$ or both non-cut-sides facing $B$. 
We then have either $D_m\subseteq B$ or $D_m\cap B = \emptyset$. 

Suppose $D_m\subseteq B$. This implies that before the surgery, $T_e$ bounds a $3$-ball $B'$ with a tubular neighbourhood of a $1$-tangle $\tau$ removed. Note that each component of $\mathcal{P}^{-1}(B')$ is bounded, so no rod can be contained in $B'\setminus N(\tau)$.  Since $T_e$ is incompressible in $\TT^3\setminus (R_1\cup R_2 \cup \ldots \cup R_n)$, the $1$-tangle $\tau$ has to be nontrivial. However, it is not possible for any union of rods to intersect all cross sections of the tubular neighbourhood of the nontrivial $1$-tangle. Hence, $T_e$ would have a compression disc, contradicting the incompressibility of $T_e$. Therefore, we have $D_m\cap B = \emptyset$, which implies that $T_e$ bounds a solid torus $V$ before the surgery. 

\textbf{Case 1:} Suppose there exists another generator in $\pi_1(T_e)$ that is also represented by a homotopically trivial loop in $\TT^3$. It follows that any component $\tilde{T_e}$ of $\mathcal{P}^{-1}(T_e)$ is a bounded set in $\RR^3$, and exactly one component bounded by $\tilde{T_e}$ is a bounded set. 

Observe that the incompressibility of $T_e$ implies the existence of a rod in the interior of the bounded solid torus $V$, which in turn implies the existence of an infinite straight line in a bounded set bounded by $\tilde{T_e}$, a contradiction.

\textbf{Case 2:} Suppose there does not exist another generator in $\pi_1(T_e)$ that is represented by a homotopically trivial loop in $\TT^3$.  We then have another generator $[\gamma]$ in $\pi_1(T_e)$ that is nontrivial in $\pi_1(\TT^3)$; denote it by $(u,v,w)\in\ZZ^3 \cong \pi_1(\TT^3)$. 

As $T_e$ is incompressible in $\TT^3\setminus (R_1\cup R_2 \cup \ldots \cup R_n)$, the solid torus $V$ contains some rod $R_k$ which is parallel to $(u,v,w)$. Since $T_e$ is not boundary-parallel in $\TT^3\setminus (R_1\cup R_2 \cup \ldots \cup R_n)$, $V$ contains at least two rods parallel to $(u,v,w)$. But two such rods would be linearly isotopic to each other in the complement of all other rods in $\TT^3$, this contradicts the assumption. 
\end{proof}  

\begin{lemma}
\label{Lem:BoundsSolidTorus}
    Suppose $T$ is an embedded torus in $\TT^3$ and $\mu,\lambda \subset T$ represent two generators of $\pi_1(T)$. If $\mu$ is homotopically trivial and $\lambda$ is homotopically nontrivial in $\TT^3$, then $T$ bounds a solid torus in $\TT^3$.  
\end{lemma} 

\begin{proof}
    By \reflem{ExistsCompressionDisc}, there exists a compression disc $D$ for $T$ in $\TT^3$. Using an argument similar to the third and fourth paragraphs in the proof of \refprop{LINIImpliesAtoroidal}, we can show that $T$ bounds a solid torus in $\TT^3$.  
\end{proof}

\begin{proposition} \label{Prop:LINIImpliesAnannular}
Suppose there exist three linearly independent rods in $\{R_1$, $R_2$, $\ldots$, $R_n\}$ and for each pair of distinct parallel rods $R_i$ and $R_j$, they are not linearly isotopic to each other in $\TT^3\setminus ((R_1\cup R_2 \cup \ldots \cup R_n)\setminus (R_i \cup R_j))$. Then the $3$-manifold $\TT^3\setminus N(R_1\cup R_2 \cup \ldots \cup R_n)$ is anannular. 
\end{proposition} 

\begin{proof} 
Suppose on the contrary that there exists an essential annulus $A_e$ in $M \coloneqq \TT^3\setminus N(R_1\cup R_2 \cup \ldots \cup R_n)$. 

Note that a component of $\partial A_e$ cannot be homotopically trivial in $\partial M$, otherwise, we may push a disc bounded by a component of $\partial A_e$ inwards to obtain a compression disc for $A_e$, contradicting the assumption that $A_e$ is incompressible. 

\textbf{Case 1:} Suppose $\partial A_e$ lies in two distinct boundary components  $\partial N(R_k)$ and $\partial N(R_l)$ of $M$.  By \cite[Lemma 5.12]{HuiPurcell}, $R_k$ and $R_l$ have to be parallel rods. 

%  We claim that the only nontrivial loops on $\partial R_k$ and $\partial R_l$ that are homotopic to each other in $\TT^3\setminus(R_k\cup R_l)$ are the $(0,1)$-curves on each rod (i.e., the circles parallel to the rod itself). 
Let $\beta_k$ be the $(k_1,k_2)$-curve in $\partial N(R_k)$ and $\beta_l$ be the $(l_1,l_2)$-curve in $\partial N(R_l)$ such that $\partial A_e = \beta_k\cup \beta_l$. Since $\beta_k$ is homotopic to $\beta_l$, their projections on a punctured plane torus perpendicular to $R_k$ are also homotopic, this implies $k_1=l_1=0$. It follows that we have $k_2=l_2=1$ as $\beta_k$ and $\beta_l$ are embedded in the tori. Hence, $\partial A_e$ is a union of two $(0,1)$-curves (i.e., longitudes) on $R_k$, $R_l$ respectively.  

Now consider a lift $\tilde{A_e}$ of $A_e$ under the covering map $\mathcal{P}\from \RR^3 \to \TT^3$. Observe that $\tilde{A_e}$ is an infinite strip with boundary components $\tilde{\beta_k}$, $\tilde{\beta_l}$ parallel to some straight-line lifts $\tilde{R_k}$, $\tilde{R_l}$ of $R_k$, $R_l$ respectively.  Let $S_p$ be the infinite plane strip 
between $\tilde{R_k}$ and $\tilde{R_l}$ in $\RR^3$. Observe that $\tilde{A_e}$, $S_p$ and strips on $\partial N(\tilde{R_k})$, $\partial N(\tilde{R_l})$ form an infinite cylinder (not necessarily embedded) in $\RR^3$. 

Since $R_k$ and $R_l$ are not linearly isotopic in the complement of other rods, the plane strip $S_p$ in the covering space $\RR^3$ must intersect the lift(s) of some other rod(s), otherwise, the projection $\calP(S_p)$ would give a linear isotopy between $R_k$ and $R_l$. 

If all lifts of rods that intersect $S_p$ were lying completely in $S_p$, then these lifts would have to be infinite straight lines parallel to $R_k$ and $R_l$, contradicting the assumption that any two parallel rods are not linearly isotopic. 
Hence, there exists a lift $\tilde{R_t}$ of a rod $R_t$ that intersects $S_p$ transversely.  

Observe that the infinite straight line $\tilde{R_t}$ intersects the infinite plane strip $S_p$ exactly once. The union of $\tilde{A_e}$ and $S_p$ forms an infinite cylinder, which projects to an immersed torus $T= A_e\cup\calP(S_p)$ in $\TT^3$ with one generator in $\pi_1(T)$ homotopically trivial in $\TT^3$ and another represented by the rod $R_k$.   

By irreducibility of $M$, we may isotope $A_e$ in $M$ such that $A_e\cap \calP(S_p)$ is a (possibly empty) union of circles that are homotopically nontrivial in $\calP(S_p)$ and $A_e\cap \calP(S_p) \cap R_t = \emptyset$. At the point where the rod $R_t$ intersects $\calP(S_p)$, $R_t$ is entering a region bounded by $T$ that is disconnected from the rest. Since $R_t$ is connected and intersects $\calP(S_p)$ exactly once, it must leave the region via $A_e$, contradicting the fact that $A_e$ is embedded in $M$. 

\textbf{Case 2:} Suppose $\partial A_e$ lies in exactly one boundary component $\partial N(R_k)$ of $M$.  The two components of $\partial A_e$ would be disjoint $(p,q)$-curves on $R_k$ for some $p, q \in \ZZ$ with  $\gcd(p,q)=1$. Observe that $\partial A_e$ separates $\partial N(R_k)$ into two annuli; denote them by $A^*$ and $A^{**}$. 

By \reflem{characteriseSSC}, $R_k$ is a $(u,v,w)$-rod associated with the parametrization stated in \refdef{Rodsss}(\ref{Def:uvwRod}) for some integral vector $(u,v,w)$ with $\gcd(u,v,w)=1$. Observe that $A_e\cup A^*$ and $A_e\cup A^{**}$ are tori embedded in $\TT^3\setminus(R_1\cup R_2\cup \ldots \cup R_n)$, and a generator of $\pi_1(A_e\cup A^*)\cong \ZZ\times\ZZ$ is represented by a $(qu,qv,qw)$-rod in $\TT^3$.

Suppose $|q|>1$.  The torus $A_e\cup A^*$ cannot be homotopic to a plane torus because $A_e\cup A^*$ is an embedded torus and its embedded $(qu,qv,qw)$-loop cannot be homotoped to an embedded loop in a plane torus. Thus, there exists a generator of $\pi_1(A_e\cup A^*)\cong \ZZ\times\ZZ$ that is trivial in $\pi_1(\TT^3)$. By \reflem{BoundsSolidTorus}, the torus $A_e\cup A^*$ bounds a solid torus $V$, which is a neighbourhood of a $(qu,qv,qw)$-loop, in $\TT^3$. Since $A_e$ is $\partial$-incompressible, there exists $\nu\in\{1,2,\ldots,n\}$ such that $R_{\nu}\subseteq V$, contradicting the fact that $R_{\nu}$ is a rod-shaped circle.

Suppose $|q|=1$. That is, $\partial A_e$ is a union of two disjoint $(p,1)$-curves on the $(u,v,w)$-rod $R_k$. We first consider the case when there exists another generator in $\pi_1(A_e\cup A^*)$ that is trivial in $\pi_1(\TT^3)$. By \reflem{BoundsSolidTorus}, the torus $A_e\cup A^*$ bounds a solid torus $V$, which is a neighbourhood of a $(u,v,w)$-loop, in $\TT^3$. Since $A_e$ is $\partial$-incompressible, there exists $\nu\in\{1,2,\ldots,n\}$ such that $R_{\nu}\subseteq V$. If $R_{\nu}\neq R_k$, this either contradicts \cite[Lemma 5.12]{HuiPurcell} or contradicts the pairwise-non-linear-isotopic condition. If $R_v=R_k$, we consider the torus $A_e\cup A^{**}$ and repeat an argument similar to the $R_{\nu}\neq R_k$ case above to come to a contradiction. 

Next, we consider the case when there does not exist any other generator in $\pi_1(A_e\cup A^*)$ that is trivial in $\pi_1(\TT^3)$. In other words, there exists another generator $[\xi]$ in $\pi_1(A_e\cup A^*)$ that is non-trivial in $\pi_1(\TT^3)$ and $[\xi] \neq (u,v,w)$ in $\pi_1(\TT^3)$. Let $(x,y,z)\in \ZZ\times\ZZ\times\ZZ\cong\pi_1(\TT^3)$ denote the homotopy class of loops that $\xi$ represents in $\TT^3$. Observe that $A_e\cup A^*$ is a non-separating torus in $\TT^3$, by an argument similar to the proof of \reflem{PlaneTorusIsEssential}, $A_e\cup A^*$ would be an essential torus.  This contradicts \refprop{LINIImpliesAtoroidal}. 

Suppose $|q|=0$. We have $|p|=1$, otherwise, $\partial A_e$ would not be embedded in $\partial N(R_k)$. Hence, $A_e\cup A^*$ bounds a solid torus $V$ in $\TT^3$. Since $A_e$ is not boundary-parallel in $M$, there exists $\sigma\in\{1,2,\ldots,n\}$ such that $R_{\sigma}$ is in the interior of $V$. This contradicts the boundedness of $V$. 

In any case, there would be contradiction. Hence, essential annulus does not exist in $M$. 
\end{proof} 

\begin{proof} [Proof of \refthm{MainThreeOrMoreRods}]
Let $n\in\ZZ$ be at least three. 
Suppose $R_1$, $R_2$, \ldots, $R_n$ are disjoint rod-shaped circles embedded in $\TT^3$. 

($\implies$): Suppose the $3$-manifold $M\coloneqq\TT^3\setminus (R_1\cup R_2 \cup \ldots \cup R_n)$ admits a complete hyperbolic structure. By \refprop{HyperbolicImpliesLI}, there exist three linearly independent rods in the set $\{R_1, R_2,\ldots, R_n\}$.  By \refprop{HyperbolicImpliesNotLinIsotopic}, no pair of distinct parallel rods $R_i$ and $R_j$ are linearly isotopic to each other in $\TT^3\setminus ((R_1\cup R_2 \cup \ldots \cup R_n)\setminus (R_i \cup R_j))$. 

($\impliedby$): 
Suppose there exist three linearly independent rods and for each pair of distinct parallel rods $R_i$ and $R_j$, they are not linearly isotopic to each other in $\TT^3\setminus ((R_1\cup R_2 \cup \ldots \cup R_n)\setminus (R_i \cup R_j))$. By \refprop{NRodsIrredAndBIrred}, the $3$-manifold $M$ is irreducible and $\partial$-irreducible. By \refprop{LINIImpliesAtoroidal}, $M$ is atoroidal. By \refprop{LINIImpliesAnannular}, $\TT^3\setminus N(R_1\cup R_2 \cup \ldots \cup R_n)$ is anannular. Therefore, by Thurston's hyperbolization theorem \cite{Thurston:3MKleinianGroupsHG}, $M$ admits a complete hyperbolic structure. 
\end{proof}

\section{Complete classification of rod complements in the $3$-torus}  \label{Sec:CompleteClassification}

With \refthm{MainThreeOrMoreRods} and the previous results in \cite{HuiPurcell}, all rod complements in the $3$-torus can be classified as hyperbolic or non-hyperbolic via some easy-to-check conditions. \refthm{Main2} states the conditions that characterise hyperbolic rod complements and those that characterise Seifert fibred ones.  

\begin{theorem}  \label{Thm:Main2}
Let $n$ be a positive integer. Suppose $R_1, R_2, \ldots, R_n$ are disjoint rod(s) in $\TT^3$. The $3$-manifold $\TT^3\setminus (R_1\cup R_2\cup \ldots\cup R_n)$ 

\begin{enumerate}
    \item \label{condition:hypIff} admits a complete hyperbolic structure if and only if there exist three linearly independent rods in $\{R_1, R_2, \ldots, R_n\}$ and for each pair of distinct parallel rods $R_i$ and $R_j$, they are not linearly isotopic in $\TT^3\setminus ((R_1\cup R_2 \cup \ldots \cup R_n)\setminus (R_i \cup R_j))$; and
    \item \label{condition:SFIff} is Seifert fibred\footnote{Abuse of terminology: 
    The statement ``$\TT^3\setminus (R_1\cup R_2\cup \ldots\cup R_n)$ is Seifert fibred'' means the compact manifold $\TT^3\setminus N(R_1\cup R_2\cup \ldots\cup R_n)$ is Seifert fibred.} if and only if $n=1$ or all rods are parallel; and 
    \item \label{condition:ToroidalIf} is toroidal if the following holds: 
    \begin{enumerate}
        \item all rod(s) span a line or plane; or 
        \item  there exist distinct integers $k,l\in\{1,2,\ldots,n\}$ such that $R_k$ and $R_l$ are linearly isotopic in $\TT^3\setminus ((R_1\cup R_2 \cup \ldots \cup R_n)\setminus (R_k \cup R_l))$. 
    \end{enumerate} 
\end{enumerate} 
\end{theorem}

\begin{proof} 
Statement (\ref{condition:hypIff}) follows from \refthm{MainThreeOrMoreRods} and \refthm{MainSingleRod_TwoRods}. 

The if-direction of Statement (\ref{condition:SFIff}) follows from \refthm{MainSingleRod_TwoRods} and the fact that the complement of parallel rods in $\TT^3$ can be foliated by circles parallel to the rods.  

Suppose the $3$-manifold $M \coloneqq \TT^3\setminus N(R_1\cup R_2\cup \ldots\cup R_n)$ has a Seifert fibring $\mathcal{W}$. Further suppose it were true that $n\neq 1$ and not all rods are parallel. There would exist linearly independent rods $R_p$ and $R_q$ with $p, q\in \{1, \ldots, n\}$. The disjoint tori $\partial N(R_p)$ and $\partial N(R_q)$ are two-sided and incompressible in $M$, by \cite[VI.34]{Jaco:3MfdTop}, $\partial N(R_p)$ and $\partial N(R_q)$ consist of fibres of $\mathcal{W}$. There would then exist regular fibres $f_p$ and $f_q$ of $\mathcal{W}$ such that $f_p \subseteq \partial N(R_p)$ and $f_q \subseteq \partial N(R_q)$. Observe that $f_p$ and $f_q$ cannot be trivial loops in $M$. Two such regular fibres of the path-connected Seifert fibred space are homotopic, which contradicts \cite[Lemma 5.12]{HuiPurcell}. As a result, $n=1$ or all rods are parallel. 

Next, we show Statement (\ref{condition:ToroidalIf}). 

Suppose all rods span a line or plane. Since the number of rods is finite, we can find a plane torus that is disjoint from all the rods in $\TT^3$. By \reflem{PlaneTorusIsEssential}, such plane torus is essential in $M=\TT^3\setminus N(R_1\cup R_2\cup \ldots\cup R_n)$. 

Suppose there exist distinct integers $k,l\in\{1,2,\ldots,n\}$ such that $R_k$ and $R_l$ are linearly isotopic in the complement of the other rods. We can surger $\partial N(R_k) \cup \partial N(R_l)$ along the annulus obtained from the linear isotopy to get an embedded torus, which is incompressible and not boundary-parallel in $M$. 

In either case, the $3$-manifold $\TT^3\setminus (R_1\cup R_2\cup \ldots\cup R_n)$ is toroidal. 
\end{proof}

%%%%%%%%%%%%%%%%%%%%%%%%%%%%%%%%%%%%%%%%%%%%%%%%%%%%%%%%%%%%%%%%% NEW Section %%%%%%%%%%%%%%%%%%%%%%%%%%%%%%%%%%%%%%%%%%%%%%%%%%%%%%%%%%%%%%%%

\section{Further discussion} \label{Sec:Application}
With the help of SnapPy \cite{SnapPy}, Hui and Purcell showed in \cite{HuiPurcell} that five of the six rod packings in \cite{OKeeffeEtAl:CubicRodPackings} each admits a complete hyperbolic structure. \refthm{MainThreeOrMoreRods} implies that all the six rod packings in \cite{OKeeffeEtAl:CubicRodPackings}, including the $8$-component $\Sigma^*$ rod packing structure omitted in \cite{HuiPurcell}, admit complete hyperbolic structures.  

\refthm{Main2} characterises all hyperbolic rod complements in $\TT^3$. An example is $\TT^3\setminus(R_x\cup R_y\cup R_z)$ shown in \cite[Figure 1]{HuiPurcell}. This hyperbolic rod complement is homeomorphic to the Borromean rings complement, and its Kleinian group can be generated by the holonomy isometries obtained from the octahedral decomposition in \cite[Figure 3]{HuiPurcell}.  Given any hyperbolic rod complements, we can input the link complements in $\TT^3$ into SnapPy \cite{SnapPy} using the method outlined in \cite{HuiPurcell}, and observe their geometries such as cusp shapes and volumes. We wonder if there are more direct ways to find hyperbolic structures of rod complements, without encoding it as a link in $\SS^3$. 

Observe that there are non-homeomorphic link complements in $\SS^3$ that share the same hyperbolic volume. It is interesting to ask whether there are non-homeomorphic rod complements in $\TT^3$ that have the same hyperbolic volume. Or would volume be a complete invariant for hyperbolic rod complements in $\TT^3$? 

% We are also interested in knowing how hyperbolic volumes behave when the parameters (modulo mapping class group of $\TT^3$) of the rods change.  

%  Suppose we have two rod complements in $\TT^3$ that admit complete hyperbolic structures. Observe that any hyperbolic rod complement is homeomorphic to the interior of a compact $3$-manifold with at least three torus boundary components, by a consequence of the thick-thin decomposition (Theorem 5.24  in \cite{Purcell:HyperbolicKnotTheory}), such rod complements have finite hyperbolic volumes. If their hyperbolic volumes are distinct, Mostow-Prasad Rigidity Theorem tells us their fundamental groups cannot be isomorphic and thus they are not homeomorphic. Hence for any two hyperbolic rod complements, we may tell whether they are topologically distinct just by comparing two real numbers. 

% This leads to a natural question that we are interested in and yet we do not have an answer: Are there relatively neat conditions that characterise a general hyperbolic link in the $3$-torus or other ambient manifolds? 

%%%%%%%%%%%%%%%%%%%%%%%%%%%%%%%%%%%%%%%%%%%%%%%%%%%%%%%%%%%%%%%%%%%%%%%%%%%%%%%%%%%%%%%%%%%%%%%%%%%%%%%%%%%%%%%%%%%%%%%%%%%%%%%%%%%%%%

%% %% To make the bibliography:  
\bibliographystyle{amsplain}  %% Uses AMS format for bibliography
%% Put all the bib entries in a file references.bib
\bibliography{ClassifyT3Rods_ref.bib}

\end{document}